%% file: main.tex
\documentclass[paper=a4, fontsize=11pt]{scrartcl} % A4 paper and 11pt font size

\usepackage[T1]{fontenc} % Use 8-bit encoding that has 256 glyphs
\usepackage{fourier} % Use the Adobe Utopia font for the document - comment this line to return to the LaTeX default
\usepackage[english]{babel} % English language/hyphenation

\usepackage{caption}
\usepackage{subcaption}
\usepackage{graphicx}
\usepackage{enumerate}% http://ctan.org/pkg/enumerate

\usepackage[]{hyperref}  %link collor

\usepackage{tikz}
\usetikzlibrary{arrows, calc,decorations.pathreplacing,decorations.markings}
\tikzset{vertex/.style = {
    draw,
    thick,
    circle,
    fill=black,
    inner sep = 0,
    minimum size=3mm,
    node distance = 2cm
}}

\usetikzlibrary{snakes}

\tikzset{
  on each segment/.style={
    decorate,
    decoration={
      show path construction,
      moveto code={},
      lineto code={
        \path [#1]
        (\tikzinputsegmentfirst) -- (\tikzinputsegmentlast);
      },
      curveto code={
        \path [#1] (\tikzinputsegmentfirst)
        .. controls
        (\tikzinputsegmentsupporta) and (\tikzinputsegmentsupportb)
        ..
        (\tikzinputsegmentlast);
      },
      closepath code={
        \path [#1]
        (\tikzinputsegmentfirst) -- (\tikzinputsegmentlast);
      },
    },
  },
  mid arrow/.style={postaction={decorate,decoration={
        markings,
        mark=at position .5 with {\arrow[#1]{stealth}}
      }}},
}

\tikzset{directed/.style = {postaction={on each segment={mid arrow=black}}
}}

\usepackage[amsthm]{ntheorem}
\usepackage{subcaption}

\usepackage{nccmath}

\newcommand{\Mod}[1]{\ (\mathrm{mod}\ #1)}

\theoremstyle{plain}

\newtheorem{thm}{Theorem} % reset theorem numbering for each chapter

\theoremstyle{definition}
\newtheorem{defn}[thm]{Definition} % definition numbers are dependent on theorem numbers

\theoremstyle{proposition}
\theoremstyle{conjecture}
\theoremstyle{lemma}
\newtheorem{lema}[thm]{Lemma} % definition numbers are dependent on theorem numbers

\theoremstyle{corolario}
\theoremstyle{observacao}
\newtheorem{obs}[thm]{Observation} % definition numbers are dependent on theorem numbers

\theoremstyle{problema}
\usepackage{sectsty} % Allows customizing section commands
\allsectionsfont{\centering \normalfont\scshape} % Make all sections centered, the default font and small caps

\usepackage{fancyhdr} % Custom headers and footers
\numberwithin{equation}{section} % Number equations within sections (i.e. 1.1, 1.2, 2.1, 2.2 instead of 1, 2, 3, 4)
\numberwithin{figure}{section} % Number figures within sections (i.e. 1.1, 1.2, 2.1, 2.2 instead of 1, 2, 3, 4)
\numberwithin{table}{section} % Number tables within sections (i.e. 1.1, 1.2, 2.1, 2.2 instead of 1, 2, 3, 4)

\newcommand{\horrule}[1]{\rule{\linewidth}{#1}} % Create horizontal rule command with 1 argument of height

\title{	
\normalfont \normalsize 
\textsc{University of Campinas} \\ [25pt] % Your university, school and/or department name(s)
\horrule{0.5pt} \\[0.4cm] % Thin top horizontal rule
\huge  Two novel results on the existence of $3$-kernels in digraphs \\ % The assignment title
\horrule{2pt} \\[0.5cm] % Thick bottom horizontal rule
}

\author{Alonso Ali and Orlando Lee} % Your name

\date{\normalsize\today} % Today's date or a custom date

\begin{document}
%\nocite{*}
\maketitle % Print the title

\newpage
\begin{abstract}

{\noindent{\Large Abstract} }\\

Let $D$ be a digraph. We call a subset $N$ of $V(D)$ \textbf{$k$-independent} if for every pair of vertices $u,v \in N$, $d(u,v) \geq k$; and we call it \textbf{$\ell$-absorbent} if for every vertex $u \in V(D) \setminus N$, there exists $v \in N$ such that $d(u,v) \leq \ell$. A \textbf{$(k,\ell)$-kernel} of $D$ is a subset of vertices which is $k$-independent and $\ell$-absorbent. A \textbf{$k$-kernel} is a $(k,k-1)$-kernel.

In this report, we present the main results from our master's research regarding kernel theory. We prove that if a digraph $D$ is strongly connected and every cycle $C$ of $D$ satisfies:
\begin{enumerate}[(i)]
    \item if $C \equiv 0 \Mod 3$, then $C$ has a short chord,
    \item if $C \not \equiv 0 \Mod 3$, then $C$ has three short chords: two consecutive and a third crossing one of the former,
\end{enumerate} then $D$ has a $3$-kernel. Moreover, we introduce a modification of the substitution method, proposed by Meyniel and Duchet in 1983, for $3$-kernels and use it to prove that a quasi-$3$-kernel-perfect digraph $D$ is $3$-kernel-perfect if every circuit of length not dividable by three has four short chords. \\

    
\end{abstract}

\newpage
\tableofcontents

%----------------------------------------------------------------------------------------
%	Section 1
%----------------------------------------------------------------------------------------

\input{1-Section.tex}

%----------------------------------------------------------------------------------------
%	Section 2
%----------------------------------------------------------------------------------------

\input{2-Section.tex}

\input{3-Section.tex}

\input{4-Section.tex}

%bibliography
\bibliographystyle{abbrv}
\bibliography{sample.bib}

\end{document}

%% file: 1-Section.tex
\newpage
\section{Introduction}

In this report, every digraph is loopless, without multiple edges and finite. The vertex set of a digraph $D$ is denoted by $V(D)$ and its arc set by $A(D)$. For every subdigraph $H$ of $D$, $D[H]$ is the subdigraph of $D$ induced by $V(D)$, and $D-H$ is the subdigraph of $D$ induced by $V(D) \setminus V(H)$. All walks, paths, circuits and cycles are considered to be directed. A digraph $D$ is \textbf{strongly connected} if for every pair of vertices $u,v \in V(D)$, there exists a $uv$-path in $D$. We call a cycle \textbf{even} (resp. \textbf{odd}) if its length is even (resp. odd). For undefined notation, we refer the reader to \cite{livro}.

A \textbf{chord} of a cycle $C = (c_0,\dots,c_{n-1},c_0)$ is an arc $a = (c_i,c_j)$, where $c_i,c_j \in V(C)$ but $a \not \in A(C)$. The \textbf{length} of the chord $a$ is defined as the distance from $c_i$ to $c_j$ in $C$. If the length of a chord is two, we call it a \textbf{short chord}. Let $\alpha = (\alpha_1, \alpha_2, \dots, \alpha_m)$ be a sequence of chords in $C$. Two chords $\alpha_i = (c_j,c_{j+k})$ and $\alpha_{i'} = (c_{j'},c_{j'+k'})$ are \textbf{crossed} if $j< j' < j+k < j'+k'$ (notation modulo $n$). The chords in $\alpha$ are \textbf{crossed} if, for every $i < m$, $\alpha_i$ crosses $\alpha_{i+1}$. Also, we say that $\alpha_i$ and $\alpha_j$ are \textbf{consecutive} if $j' = j+k$ (notation modulo $n$); and the chords in $\alpha$ are consecutive if, for every $i <m$, $\alpha_{i+1}$ is consecutive to $\alpha_i$.

A \textbf{kernel} of a digraph $D$ is a set $K \subseteq V(D)$ which is independent in $D$ and for every vertex $u \in V(D) \setminus K$, there exists an arc $(u,v)$, where $v \in K$. A digraph is \textbf{kernel-perfect} if every induced subdigraph has a kernel. We call a subset $N$ of $V(D)$ \textbf{$k$-independent} if for every pair of vertices $u,v \in N$, $d(u,v) \geq k$; and we call it \textbf{$\ell$-absorbent} if for every vertex $u \in V(D) \setminus N$, there exists $v \in N$ such that $d(u,v) \leq \ell$. A \textbf{$(k,\ell)$-kernel} of $D$ is a subset of vertices which is $k$-independent and $\ell$-absorbent. A \textbf{$k$-kernel} is a $(k,k-1)$-kernel. 

The concept of a kernel was introduced by von Neumann and Morgenstern in 1944 \cite{GAMETHEORY}, in the context of game theory, to model social and economic interactions. In light of its relation to the Strong Perfect Graph Conjecture (now the Strong Perfect Graph Theorem \cite{THESTRONGPERFECTGRAPHTHEOREM}), kernel theory gained a lot of attention and was thoroughly researched. When Kwaśnik proposed the concept of $k$-kernels \cite{KLKERNELS} and generalized Richardson's Theorem for $k$-kernels \cite{GENERALIZATIONOFRICHARDSONTHEOREM}, $k$-kernels became an interesting line of study in kernel theory.

In this report, we present the main results from our master's research regarding kernel theory. In Section 2, we prove that if a digraph $D$ is strongly connected and every cycle $C$ of $D$ satisfies:
\begin{enumerate}[(i)]
    \item if $C \equiv 0 \Mod 3$, then $C$ has a short chord,
    \item if $C \not \equiv 0 \Mod 3$, then $C$ has three short chords: two consecutive and a third crossing one of the former,
\end{enumerate} then $D$ has a $3$-kernel. In Section 3, we introduce a modification of the substitution method for $3$-kernels and use it to prove that a quasi-$3$-kernel-perfect digraph $D$ is $3$-kernel-perfect if every circuit of length not dividable by three has four short chords.

%% file: 2-Section.tex
\newpage
\section{A sufficient condition for the existence of $3$-kernels in digraphs}
\label{chapter2}

Let $D$ be a strongly connected digraph where every cycle $C$ satisfies:
\begin{enumerate}[(i)]
    \item if $C \equiv 0 \Mod 3$, then $C$ has a short chord,
    \item if $C \not \equiv 0 \Mod 3$, then $C$ has three short chords: two consecutive and a third crossing one of the former.
\end{enumerate}
In this section, we prove $D$ has a $3$-kernel.

An useful tool for demonstrating the existence of $3$-kernels in digraphs is Lemma \ref{lema:kkernelssekernel}, which states that a digraph has a $3$-kernel if, and only if, its $2$-closure has a kernel. Before we present Lemma \ref{lema:kkernelssekernel}, we introduce the definition of the $k$-closure of a digraph.

\begin{defn}
Let $D$ be a digraph. The \textbf{$k$-closure} of $D$, denoted by $C^k(D)$, is the digraph $D'$, where $V(D')=V(D)$ and $(u,v) \in A(D')$ if $d_D(u,v) \leq k$. Figure \ref{fig:digrafoe2fecho} depicts a digraph and its $2$-closure.
\end{defn}

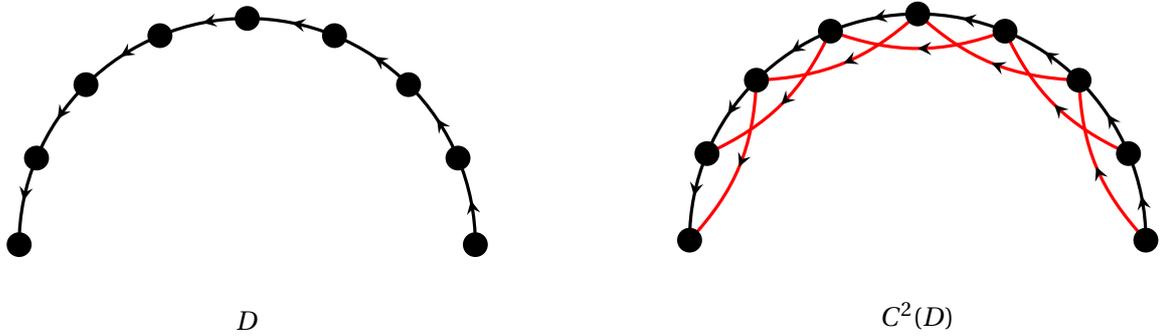
\begin{figure}[h]
    \centering
    \begin{subfigure}[b]{.4\textwidth}
        %\raisebox{ 20mm }
      \centering
      \begin{tikzpicture}[line width = 1.2]
            \def \n {8}
                      \def \namearray{{meem}}
                      \def\names{{"v_1","v_2","v_3","w_1","w_2","w_3","u_1","u_2","u_3"}}
                     
                      \def \radius {3cm}
                      \def \margin {0} % margin in angles, depends on the radius
                        
                        \foreach \s in {1,...,\n}
                        {
                          \node[vertex] (\s) at ({180/\n * (\s - 1)}:\radius) {};
                          
                          \draw[directed] ({180/\n * (\s - 1)+\margin}:\radius) 
                            arc ({180/\n * (\s - 1)+\margin}:{180/\n * (\s)-\margin}:\radius);
                        }
                        
                        \node[vertex] (\n+1) at (180:\radius) {};
                        
                        \node at (0, -1){$D$};
      \end{tikzpicture}
    \end{subfigure} \hfill 
    \begin{subfigure}[b]{.4\textwidth}
      \centering
      \begin{tikzpicture}[line width = 1.2]
            \def \n {8}
                      \def \namearray{{meem}}
                      \def\names{{"v_1","v_2","v_3","w_1","w_2","w_3","u_1","u_2","u_3"}}
                     
                      \def \radius {3cm}
                      \def \margin {0} % margin in angles, depends on the radius
                        
                        \foreach \s in {1,...,\n}
                        {

                          \draw[directed] ({180/\n * (\s - 1)+\margin}:\radius) 
                            arc ({180/\n * (\s - 1)+\margin}:{180/\n * (\s)-\margin}:\radius);
                        
                          \ifthenelse{\not \s = \n}{
                          \draw[directed,red] ({180/\n * (\s - 1)}:\radius) to [bend left= 20] ({180/\n * (\s +1)-\margin}:\radius);
                          };
                          
                          \node[vertex] (\s) at ({180/\n * (\s - 1)}:\radius) {};

                        }
                        
                        \node at (0, -1){$C^2(D)$};
                        
                        \node[vertex] (\n+1) at (180:\radius) {};
      \end{tikzpicture}
    \end{subfigure}
    \caption{An example of a digraph $D$ and its $2$-closure $C^2(D)$. Those arcs in $C^2(D)$ which are not in $D$ are painted red.}
    \label{fig:digrafoe2fecho}
\end{figure}

\begin{lema}[\cite{galeana2014existence}]
\label{lema:kkernelssekernel}
Let $k \geq 3$ be an integer. Let $D$ be a digraph and let $K \subseteq V(D)$. The subset of vertices $K$ is a $k$-kernel of $D$ if, and only if, $K$ is a kernel of $C^{(k-1)}(D)$.
\end{lema}

A particularly useful theorem was proved by Duchet in 1980.
\begin{thm}[\cite{DUCHET}]
\label{thm:todocircuitoarcosimetrico}
If every cycle of a digraph $D$ has a symmetric arc, then $D$ is kernel-perfect.
\end{thm}

The strategy of the proof to our theorem is to show that the $2$-closure of a digraph whose cycles satisfies conditions (i) and (ii) has a symmetric arc in every cycle. From Theorem \ref{thm:todocircuitoarcosimetrico}, the $2$-closure of such digraph has a kernel. Therefore, by Lemma \ref{lema:kkernelssekernel}, the digraph has a $3$-kernel.

Before we present the demonstration, we must introduce the main lemma used in the proof. 

\begin{lema}
\label{lema:doresultadobom}
Let $D$ be a strongly connected digraph where every cycle $C$ of $D$ satisfies:

\begin{itemize}
    \item if $|C| \equiv 0 \Mod 3$, then $C$ has a short chord,
    \item if $|C| \not \equiv 0 \Mod 3$, then $C$ has two consecutive short chords.
\end{itemize} Then, for every $(u,v) \in A(D)$, there is a $(v,u)$-path of length at most two in $D$.
\end{lema}

\begin{proof}
Let $f =(u,v) \in A(D)$. Since $D$ is strongly connected, there is a minimal $(v,u)$-path $T = (t_0 = v, \dots, t_s = u)$ in $D$. Note that $C = T \cup (u,v)$ is a cycle. Clearly, if $|C| = 3$, then $|T| = 2$ and the result follows. Assume, for the sake of contradiction, that $|T| > 2$. We will show that it is not possible. Assume that $|C| \not \equiv 0 \Mod 3$. Because $T$ is minimal, the only possible short chords in $C$ are $(t_{s-1}, v)$ and $(u,t_1)$. These chords are consecutive in $C$ only if $|C| = 3$, which contradicts the length of $T$ being greater than two. Suppose then that $|C| \equiv 0 \Mod 3$. Due to the hypothesis, there is a short chord in $C$: $(t_{s-1}, v)$ or $(u,t_1)$. Let $(a,b)$ be one of such chords. Note that $(a,b) \cup (b,T,a)$ is a cycle of length $\not \equiv 0 \Mod 3$ and, therefore, must have two consecutive short chords. Analogously to the former case, the chords in this cycle are consecutive only if the cycle has length dividable by three, a contradiction. Hence, $|T| \leq 2$. $\blacksquare$
\end{proof}

\begin{thm}
\label{thm:meuteorema2cc}
Let $D$ be a strongly connected digraph. Assume that every cycle $C$ of $D$ satisfies:

\begin{itemize}
    \item if $|C| \equiv 0 \Mod 3$, then $C$ has a short chord,
    \item if $|C| \not \equiv 0 \Mod 3$, then $C$ has three short chords: two consecutive and one crossing one of the former.
\end{itemize} Then $D$ has a $3$-kernel.
\end{thm}

\begin{proof}
Let $D'= C^2(D)$. It follows from Lemma \ref{lema:kkernelssekernel} that $D$ has a $3$-kernel if $D'$ has a kernel. We will show that every cycle in $D'$ has a symmetric arc and, by Theorem \ref{thm:todocircuitoarcosimetrico} $D'$ has a kernel.

Assume that there exists a cycle $C = (c_0, c_1, \dots, c_{n-1}, c_0)$ in $D'$ with no symmetric arc. Choose such cycle $C$ with the shortest length in $D'$. Note that no arc in $C$ exists in $D$. In virtue of Lemma \ref{lema:doresultadobom}, if $(c_i,c_{i+1})\in A(D)$, then $(c_{i+1},c_i) \in A(D')$. Therefore, $d_D(c_i,c_{i+1}) = 2$, for every $i < n$. 

Let $C'$ be the closed trail resulting from the substitution of every $(c_i,c_{i+1})$ in $C$ for $(c_i,c_{i,(i+1)}, c_{i+1})$. Figure \ref{fig:caminhoaumentado} illustrates an example of $C$ and $C'$. We will show that $C'$ is a cycle. Assume that $c_{i,(i+1)} = c_{j,(j+1)}$, $i \neq j$. Then, $d_D(c_i,c_{j+1}) = 2$ and $(c_i,c_{j+1}) \in A(D')$. Let $C''= (c_{j+1}, C, c_i) \cup (c_i,c_{j+1})$. Note that $C''$ is a cycle in $D'$ and $|C''| < |C|$, which contradicts the minimality of $C$. Hence, $c_{i,(i+1)} \neq c_{j,(j+1)}$ for every $i \neq j$. Since $C$ is a cycle, $c_i \neq c_j$ for every $i \neq j$. It follows that $C'$ is a cycle in $D$. Furthermore, note that $C'$ is an even cycle.

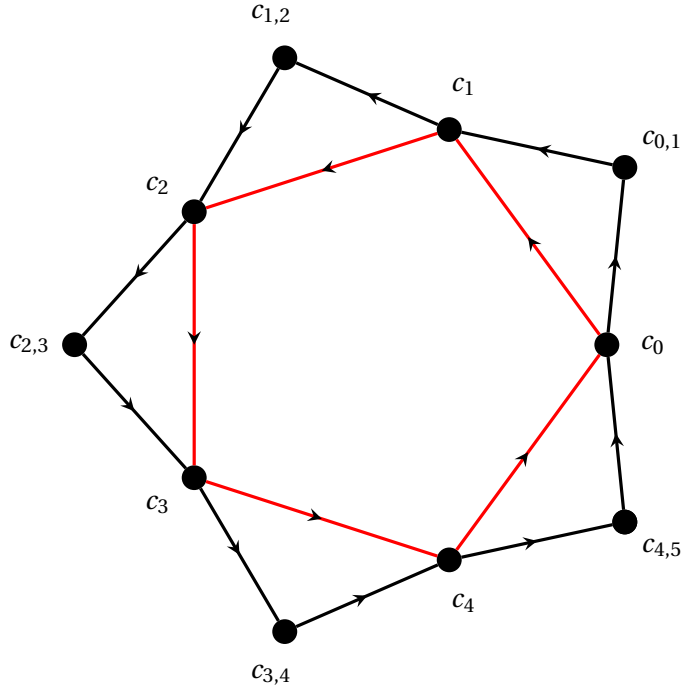
\begin{figure}[h]
    \centering
    \begin{tikzpicture}[line width = 1.2]
            \def \n {5}
                      \def \radius {3cm}
                      \def \margin {0} % margin in angles, depends on the radius
                        
                        \foreach \s in {0,...,4}
                        {
                          \node[vertex] (\s) at ({360/\n * (\s)}:\radius) {};
                          
                          \node at ({360/\n * (\s)}:\radius+0.6cm) {$c_{\s}$};
                          
                        }
                        
                        \foreach \s in {5,...,10}
                        {
                          \node[vertex] (\s) at ({360/\n * (\s - 1/2)}:\radius + 1cm) {};
                          
                        }
                        
                        \node at ({360/\n * (6 - 1/2)}:\radius + 1.6cm) {$c_{0,1}$};
                        \node at ({360/\n * (7 - 1/2)}:\radius + 1.6cm) {$c_{1,2}$};
                        \node at ({360/\n * (8 - 1/2)}:\radius + 1.6cm) {$c_{2,3}$};
                        \node at ({360/\n * (9 - 1/2)}:\radius + 1.6cm) {$c_{3,4}$};
                        \node at ({360/\n * (5 - 1/2)}:\radius + 1.6cm) {$c_{4,5}$};

                        \draw[directed,red] (0) to (1);
                        \draw[directed,red] (1) to (2);
                        \draw[directed,red] (2) to (3);
                        \draw[directed,red] (3) to (4);
                        \draw[directed,red] (4) to (0);
                        
                        \draw[directed] (0) to (6);
                        \draw[directed] (6) to (1);
                        \draw[directed] (1) to (7);
                        \draw[directed] (7) to (2);
                        \draw[directed] (2) to (8);
                        \draw[directed] (8) to (3);
                        \draw[directed] (3) to (9);
                        \draw[directed] (9) to (4);
                        \draw[directed] (4) to (5);
                        \draw[directed] (5) to (0);

      \end{tikzpicture}
    \caption{An illustration of a cycle $C$, in red, and $C$, in black.}
    \label{fig:caminhoaumentado}
\end{figure}

If $|C'| \not \equiv 0 \Mod 3$, then $C$ has two consecutive short chords and a crossed short chord. Note that $C'$ can be denoted as $(c'_0, c'_1, \dots, c'_{2n-1}, c'_0)$, where $c'_{2i} \in C$ for every $i < n$. Therefore, a short chord in $C'$ connects two vertices with the same parity. Clearly, one of the hypothesis' chords must connect two vertices of $C$. Since this chord exists in $D$, Lemma \ref{lema:doresultadobom} guarantees that the chord is symmetric in $D'$. Therefore, $C$ has a symmetric arc in $D'$. Assume that $|C'| \equiv 0 \Mod 3$. By the hypothesis, $C'$ has a short chord. If the short chord is $(c'_{2k},c'_{2k+2})$, for any $k < n$, then if follows from Lemma \ref{lema:doresultadobom} that $C$ has a symmetric chord. Assume this chord is
$(c'_{2k-1}, c'_{2k+1})$, for some $k < n$ (notation modulo $2n$). Then $B = (c'_{2k+1},C', c'_{2k-1}) \cup (c'_{2k-1}, c'_{2k+1})$ is a cycle of length $\not \equiv 0 \Mod 3$ and has two consecutive short chords and a crossed short chord. Note that if $(c'_{2k-2}, c'_{2k+1})$ or $(c'_{2k-1}, c'_{2k+2})$ are chords in $B$, then there is a path of length two connecting two vertices not adjacent in $C$. This implies in the existence of a cycle without symmetric arcs and with length less than $|C|$, contradicting its minimality. Therefore, $(c'_{2k-2}, c'_{2k+1})$ or $(c'_{2k-1}, c'_{2k+2})$ can not exist in $D$, as Figure \ref{fig:duascordasimpossiveis} shows. Hence, the three short chords must be chord in the path $(c'_{2k+1}, \dots, c'_{2k-1})$. Since it is a path whose indexes alternate between even and odd numbers, one of the three short chords connects two vertices of $C$. From Lemma \ref{lema:doresultadobom}, such arc is symmetric in $D'$ and it follows that $C$ has a symmetric arc.

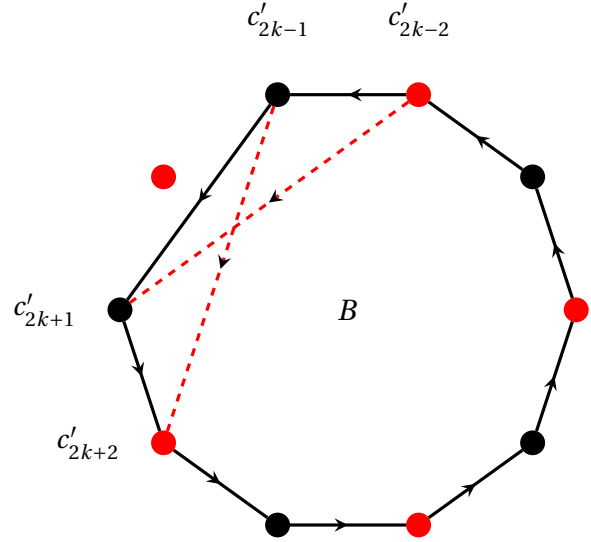
\begin{figure}[h]
    \centering
    \begin{tikzpicture}[line width = 1.2]
            \def \n {10}
                      \def \radius {3cm}
                      \def \margin {0} % margin in angles, depends on the radius
                        
                        \foreach \s in {0,2,4,6,8}
                        {
                          \node[vertex,red] (\s) at ({360/\n * (\s)}:\radius) {};
                          
                          %\node at ({360/\n * (\s)}:\radius+0.6cm) {$c_{\s}$};
                          
                        }
                        
                        \foreach \s in {1,3,5,7,9}
                        {
                          \node[vertex] (\s) at ({360/\n * (\s)}:\radius) {};
                          
                        }
                        
                        %\node at ({360/\n * (6 - 1/2)}:\radius + 1.6cm) {$c_{0,1}$};
                        %\node at ({360/\n * (7 - 1/2)}:\radius + 1.6cm) {$c_{1,2}$};
                        %\node at ({360/\n * (8 - 1/2)}:\radius + 1.6cm) {$c_{2,3}$};
                        %\node at ({360/\n * (9 - 1/2)}:\radius + 1.6cm) {$c_{3,4}$};
                        %\node at ({360/\n * (5 - 1/2)}:\radius + 1.6cm) {$c_{4,5}$};

                        \draw[directed] (0) to (1);
                        \draw[directed] (1) to (2);
                        \draw[directed] (2) to (3);
                        %\draw[directed,dashed] (3) to (4);
                        %\draw[directed,dashed] (4) to (5);
                        \draw[directed] (5) to (6);
                        \draw[directed] (6) to (7);
                        \draw[directed] (7) to (8);
                        \draw[directed] (8) to (9);
                        \draw[directed] (9) to (0);
                        
                        \draw[directed] (3) to (5);
                        \node[above of = 3] {$c'_{2k-1}$};
                        \node[left of = 5] {$c'_{2k+1}$};
                        \node[above of = 2] {$c'_{2k-2}$};
                        \node[left of = 6] {$c'_{2k+2}$};
                        
                        \draw[directed, red,dashed] (2) to (5);
                        \draw[directed, red,dashed] (3) to (6);
                        
                        \node {\large $B$};

      \end{tikzpicture}
    \caption{An example of the cycle $B$. The vertices painted red are part of $C$. If one of the dashed arcs in red exists, then $d_D(c'_{2k-2},c'_{2k+2}) = 2$.}
    \label{fig:duascordasimpossiveis}
\end{figure}

Since every cycle of $D'$ has at least one symmetric arc, it follows from Theorem \ref{thm:todocircuitoarcosimetrico} that $D$ has a kernel. Hence, from Lemma \ref{lema:kkernelssekernel}, $D$ has a $3$-kernel. $\blacksquare$

\end{proof}

%% file: 3-Section.tex
\newpage
\section{The substitution method for $3$-kernels}
\label{chapter3}

In 1983, Duchet and Meyniel \cite{DUCHETDUASCORDAS} presented Theorem \ref{thm:duascordascurtascruzadas} regarding kernel-perfect digraphs.
\begin{thm}[Duchet and Meyniel, 1983]
\label{thm:duascordascurtascruzadas}
If $D$ is a digraph and every odd circuit has two crossed short chords, then $D$ is kernel-perfect.
\end{thm}
To prove Theorem \ref{thm:duascordascurtascruzadas}, Duchet and Meyniel introduced the \textbf{substitution method}. The substitution method is an iterative process of: removing an arbitrary vertex $x_0$; obtaining a kernel $K$ of $D-x_0$; and, in $D$, performing a sequence of substitutions of vertices in $K$ for other vertices of $D$, in order to create a kernel $K'$ of $D$ where $x_0 \in K'$. Although the method does not guarantee its result to be a kernel, the authors proved that if a digraph $D$ satisfies the hypothesis of Theorem \ref{thm:duascordascurtascruzadas}, then the result is a kernel of $D$.

In this section, we present a generalization of the substitution method for $3$-kernels and use it to prove a new theorem regarding the existence of $3$-kernels in digraphs.

\subsection{The $3$-substitution method}

The general idea of the $3$-substitution method is analogous to the original method: vertices of $D$ are divided into subsets of vertices which are then used to alter the kernel $K$ of $D-x_0$. In Duchet and Meyniel's method, vertices are divided into sets $N_i$. If the index $i$ of a set is odd, those vertices are removed from $K$; and, if the index of a set is even, those vertices are added to $K$. The resulting set $K' = (K \setminus N_{odd}) \cup N_{even}$ is called a \textbf{pre-kernel}. In our method, vertices are divided into sets $N_i$ whose functions depend on the congruence of $i \bmod{3}$. Sets with indexes congruent to zero will be added to the $3$-kernel, and sets with indexes not congruent to zero will be removed from the $3$-kernel.

\subsubsection{Definitions}
\label{definicoessubsub}

Before we present the method, we need to introduce some new definitions that will be used in this section.

\begin{defn}
Let $D$ be a digraph and let $\ell \in \mathbb{N}$. The \textbf{$\ell$-in-neighbourhood} of a vertex $v \in V(D)$ is the set $N^{\ell-}(v) = \{ u \in V(D) : d(u,v) = \ell \}$. The $\ell$-in-neighbourhood of a set of vertices $S \subseteq V(D)$ is $N^{\ell-}(S) = \{ u \in V(D) \setminus S : d(u,v) = \ell \textrm{, for some } v\in S\}$. Note that $N^{1-}(S) = N^-(S)$.
\end{defn}

\begin{defn}
let $D$ be a digraph, let $S \subseteq V(D)$ and let $\ell \in \mathbb{N}$. The \textbf{cone} of distance $\ell$ of $S$ is the set $\Delta^{\ell +}(S)= \{ v \in V(D) : \textrm{for some $u \in S$, } 0 < d(u,v) \leq \ell \}$.
\end{defn}

\subsubsection{The $3$-substitution sequence and pre-$3$-kernel}

Using the definitions from Section \ref{definicoessubsub} we define a \textbf{$3$-substitution sequence} and a \textbf{pre-$3$-kernel}, some of the key concepts of the $3$-substitution method. 

\begin{defn}
let $D$ be a digraph. A \textbf{$3$-substitution sequence} starting at a vertex $x_0 \in V(D)$ and a $3$-kernel $K$ of $D-x_0$ is a sequence of sets $(N_0, N_1, \dots, N_{3p})$ built from the following process:

\begin{fleqn}
    \begin{equation}
        N_0 = M_0 = \{ x_0\}
    \end{equation}
    
    \begin{equation}
        N_{3k+1} = (N^-(N_{3k}) \cap K) \setminus \cup_{k' < k} \left ( N_{3k'+1} \cup N_{3k'+2}\right )
    \end{equation}
    
   \begin{equation}
       N_{3k+2} = (N^{2-}(N_{3k}) \cap K) \setminus \cup_{k' < k} \left ( N_{3k'+1} \cup N_{3k'+2} \right )
   \end{equation}
   
   \begin{multline}
       M_{3k+3} =  \{  x \in V \left( D \setminus \cup_{k'\leq k} M_{3k'} \right) : \Delta^2(x) \cap K \subseteq (\cup_{k' \leq k} N_{3k'+1} \cup N_{3k'+2}) \textrm{ and } \\ \Delta^2(x) \cap \cup_{k' \leq k} N_{3k'} = \emptyset  \}
   \end{multline}
   
   \begin{equation}
       N_{3k+3} \textrm{ is a $3$-kernel of } D[M_{3k+3}]
   \end{equation}
   
   \begin{equation}
       \textrm{$p$ is the smallest integer $k$ such that $N_{3k+1} = \emptyset$ and $N_{3k+2} = \emptyset$}
   \end{equation}
\end{fleqn}

\end{defn}

\begin{obs}
Note that the sets in a $3$-substitution sequence are disjoint: from the definition of each set, a vertex can not belong to two sets. This observation is fundamental to comprehend the following definitions and lemmas in the next section.
\end{obs}

\begin{defn}
Let $D$ be a digraph. A \textbf{pre-$3$-kernel} $N$ is a set of vertices obtained from a $3$-substitution sequence $(N_0,N_1,\dots,N_{3p})$ from a vertex $x_0 \in V(D)$ and a $3$-kernel $K$ of $D-x_0$. Formally, $$N = \left( K \setminus \bigcup_{k=0}^p N_{3k+1}\cup N_{3k+2}\right) \cup \bigcup_{k=0}^p N_{3k}.$$
\end{defn}

By definition, $N_1 \subseteq N^-(x_0) \cap K$ and $N_2 \subseteq N^{2-}(x_0) \cap K$. The idea is that, from demanding that $x_0$ belongs to the pre-$3$-kernel the method builds, vertices in $N_1$ and $N_2$ must be removed from $K$, in order to keep the pre-$3$-kernel $3$-independent. Now, consider the set $M_3 = \{ x \in V(D) \setminus (N_1 \cup N_2) : \Delta^{2+}(x) \cap K \subseteq N_1 \cup N_2 \}$. Note that a vertex in $M_3$ does not belong to $K$ by definition. Furthermore, vertices of $M_3$ are $2$-absorbed by $N_1 \cup N_2 \subseteq K$ in $D$. Since $M_3$ may not be $3$-independent, we add to the pre-$3$-kernel only the vertices of a $3$-kernel $N_3$ from the subdigraph $D[M_3]$. 

The method's next iterations follow the same idea: build subsets of vertices $$N_1, N_2, N_4, N_5, \dots, N_{3k+1}, N_{3k+2}$$ of $K$, whose vertices will be removed from $K$, and subsets $$ N_0, N_3, N_6, \dots, N_{3k}$$ of $V\setminus K$, whose vertices will be added to the pre-$3$-kernel. Intuitively, a pre-$3$-kernel is a set composed from the removal of ``unwanted'' vertices of $K$ and the addition of vertices that ``correct'' the $2$-absorption of $K$ in $D$. As we explain in the next section, a pre-$3$-kernel is not guaranteed to be a $3$-kernel. Nevertheless, we show that a pre-$3$-kernel is $2$-absorbent and, if it is not $3$-independent, there are certain paths between two vertices in the pre-$3$-kernel. Before we present the properties of a pre-$3$-kernel in the next section, we define \textbf{intermediate vertices}.

\begin{defn}
Let $D$ be a digraph and $S = (N_0, \dots ,N_{3p})$ a $3$-substitution sequence of $D$. A sequence of \textbf{intermediate sets} of $S$ is a sequence of sets $I = (N'_1, N'_2, N'_4, \dots, N'_{3p-2}, N'_{3p-1})$ defined as:
\begin{fleqn}
\begin{equation}
    N'_{3k+1} = N^-(N_{3k}) \setminus N_{3k+1}, \textrm{ e }
\end{equation}

\begin{equation}
    N'_{3k+2} = N^{2-}(N_{3k}) \setminus N_{3k+2},
\end{equation}
\end{fleqn} for every $k < p$.
If a vertex belongs to an intermediate set, then it is called an \textbf{intermediate vertex}.
\end{defn}

Let $D$ be a digraph, let $S$ be a $3$-substitution sequence of $D$ and let $S'$ be a sequence of intermediate sets of $S$. Intuitively, a vertex $x \in V(D)$ belongs to an intermediate set of $S'$ if $x$ is not in any set of $S$ but is ``between'' two vertices that belong to any set of $S$. For example, let $v \in N_1$ and $w \in N_3$. By the definition of $N_3$, $v$ belongs to the cone of distance $2$ of $w$. Therefore, there exists a $(w,v)$-path $T = (w,t_1, v)$ with length two. The vertex $t_1$ can not belong to any set in the form $N_{3k+1}$ or $N_{3k+2}$, for any $k$, but is ``used'' by $u$ to get to $v$. Hence, we say that $t_1$ is intermediate to $N_1$ and $N_3$. So, it is in $N'_2$. Figure \ref{fig:exemplointermediario} depicts an example.

\begin{figure}
    \centering
    \begin{tikzpicture}
                \node[vertex] (0) {};
                \node[vertex, left of = 0] (1) {};
                \node[vertex, left of = 1] (2) {};
                \node[vertex, left of = 2] (3) {};
                \node[vertex, left of = 3] (4) {};
                \node[vertex, left of = 4] (5) {};
                
                \node[above of = 0, yshift=-0.2cm] {$N_0$};
                \node[above of = 1, yshift=-0.2cm] {$N_{1}$};
                \node[above of = 2, yshift=-0.2cm] {$N'_2$};
                \node[above of = 3, yshift=-0.2cm] {$N_3$};
                \node[above of = 4, yshift=-0.2cm] {$N'_4$};
                \node[above of = 5, yshift=-0.2cm] {$N_5$};
                
                %\node[right of = -3] (rd) {$\dots$};
                \node[left of = 5] (ld) {$\dots$};
                
                \draw[directed] (ld) to (5);
                %\draw[directed] (-3) to (rd);

                \draw[directed] (5) to (4);
                \draw[directed] (4) to (3);
                \draw[directed] (3) to (2);
                \draw[directed] (2) to (1);
                \draw[directed] (1) to (0);

    \end{tikzpicture}
    \caption{The vertices in $N'_2$ and $N'_4$ are, respectively, at distance two and four from $x_0 \in N_0$. Since they do not belong to the $3$-kernel of $D-x_0$, they are called intermediate.}
    \label{fig:exemplointermediario}
\end{figure}
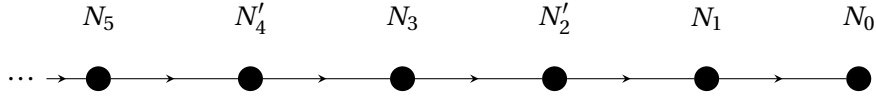

\subsection{Properties of a pre-$3$-kernel}

In this section, we present some useful lemmas about the properties of a pre-$3$-kernel. 

\begin{lema}
\label{lema:kkernelabsorventequaseindependente}
Let $D$ be a digraph, let $Z=(N_0,N_1,\dots, N_{3p})$ be a $3$-substitution sequence of $D$ starting at $x_0 \in V(D)$ and a $3$-kernel $K \subseteq D-x_0$, and let $N$ be the pre-$3$-kernel of $D$ obtained from $Z$. The pre-$3$-kernel $N$ is $2$-absorbent in $D$. Furthermore, if there exists a $(N,N)$-path $T = (a, \dots, b)$ in $D$ such that $|T| \leq 2$, then $$ a \in N_{3k'} \textrm{ and } b \in N_{3k}, \textrm{ for some pair $k',k$,  } 0 \leq k' \leq k \leq p.$$
\end{lema}

\begin{proof}
Assume, for the sake of contradiction, that there is $v \in V(D) \setminus N$ such that $\Delta^2(v) \cap N = \emptyset$. By the definition of $M_{3k+3}$, every vertex not $2$-absorbed by $N$ must belong to $M_{3k+3}$, therefore, $v \in M_{3k'}$ for some $k' \geq 0$. Since $v \not \in N_{3k'}$, there exists $u \in N_{3k'}$ such that $d(v,u) \leq 2$. But then $v$ is $2$-absorbed by $N_{3k'}$ and, consequently, by $N$, a contradiction. Hence, $N$ is $2$-absorbent.

We now prove that if there is a $(N,N)$-path $T = (a, \dots, b)$ in $D$ such that $|T| \leq 2$, then $ a \in N_{3k'} \textrm{ and } b \in N_{3k}, \textrm{ for some pair $k,k'$, } 0 \leq k' \leq k \leq p.$ Let $T$ be a $(N,N)$-path with $|T| \leq 2$. If $a,b \in N$, then $a$ and $b$ belong to $K$ or $N_{3k}$, for some $k \geq 0$. Clearly, $a$ and $b$ can not both belong to $K$, since $K$ is a $3$-kernel of $D-x_0$ and, if there is a path $(a,x_0,b)$, then $a \in N_1$ and does not belong to $N$. Assume that $a \in K$ and $b \in N_{3k'}$, for some $k' \geq 0$. By the definition of $N_{3k'+1}$, if $a \not \in N_{3k'+1} \cup N_{3k'+2}$, then $a \in N_{3\ell+1} \cup N_{3\ell+2}$, for some $\ell < k'$. Both cases contradict the hypothesis that $a \in N$. Assume that $a \in N_{3k'}$, for some $k' \geq 0$ and $b \in K$. Then $a \in M_{3k'}$, which contradicts the definition of $M_{3k+3}$ since $b \in \Delta^2(a)$ and $b \in K$. Assume, for the sake of contradiction, that $a \in N_{3k}$ and $b \in N_{3k'}$, for some $k' \leq k$. The case where $k' = k$ contradicts the definition of $N_{3k}$ because it is not a $3$-kernel of $D[M_{3k}]$. Finally, assume that $k' < k$. By the definition of $M_{3k}$, $a \not \in M_{3k}$ since it already is absorbed by $b \in N_{3k'}$. Clearly, if $a \not \in M_{3k}$ then $a \not \in N_{3k}$. Therefore, if there exists $(N,N)$-path $T = (a, \dots, b)$ in $D$ with $|T| \leq 2$, then $ a \in N_{3k'} \textrm{ and } b \in N_{3k}, \textrm{ where } 0 \leq k' \leq k \leq p$, as we wanted to prove. $\blacksquare$
\end{proof}

\begin{defn}
Let $D$ be a digraph and let $(N_0,N_1,\dots,N_{3p})$ be a $3$-substitution sequence of $D$ starting at $x_0 \in V(D)$ and a $3$-kernel $K \subseteq D-x_0$. We call a path $T=(t_s, t_{s-1}, \dots, t_1, t_0)$ in $D$ a \textbf{road} if it satisfies the following conditions:

\begin{fleqn}

\begin{equation}
    t_s \in N_s, \textrm{for some $s \leq 3p$};
\end{equation}

\begin{equation}
\label{cond1}
    t_{3i+1} \in N_{3i+1} \Leftrightarrow t_{3i+2} \in N'_{3i+2}, \textrm{ for every $3i \leq s-2$ such that $t_{3i+2} \in V(T)$};
\end{equation}

\begin{equation}
\label{cond2}
    t_{3i} \in N_{3i}, \textrm{ for every $3i \leq s$};
\end{equation}

\begin{equation}
\label{cond3}
    \textrm{if }(t_i,t_{i-2}) \in A(D) \textrm{, then $t_i \in N'_{3j+1}$ and $t_{i-2} \in N'_{3(j-1)+2}$, for some $3j < s$}.
\end{equation}

\end{fleqn}

\end{defn}

\begin{lema}
\label{lema:kkernelcaminhopropriedades}
Let $D$ be a digraph and let $(N_0,N_1,\dots,N_{3p})$ be a $3$-substitution sequence of $D$ starting at $x_0 \in V(D)$ and a $3$-kernel $K \subseteq D-x_0$. For every $s \leq 3p$ and every vertex $v \in N_s$ there exists a $(v,x_0)$-path $T=(t_s = v, t_{s-1}, \dots, t_1, t_0 = x_0)$ such that $T$ is a road.
\end{lema}

\begin{proof}
Let $v \in N_s$. We first prove the existence of a $(v,x_0)$-path $T$ with length $s$ satisfying properties (\ref{cond1}) and (\ref{cond2}). After that, we prove it satisfies property (\ref{cond3}). We proceed by induction in $s$.

\textbf{Base case}: (s = 1) By the definition of $N_1$, $x_0 \in N^+(v)$. Therefore, it suffices to take $T = (v,x_0)$.

\textbf{Induction hypothesis}: Assume that for every $s'<s$ and every vertex $v \in N'_{s'}$, there exists a path $T' = (v = t_{s'}, \dots, x_0 = t_0)$ with length $s'$ satisfying properties (\ref{cond1}) and (\ref{cond2}).

\textbf{Induction step}: Let $v \in N_s$. By construction, $v \in (N^-(N'_{s-1}) \cup N^-(N_{s-1}))$. We will show that there exists a path $T$ satisfying (\ref{cond1}) and (\ref{cond2}).

\begin{itemize}
    \item If $v \in N_{3i+1}$, for some $i$, then there is $u \in N_{3i}$ such that $(v,u) \in A(D)$. From the induction hypothesis, there is a path $T' = (u = t_{s'}, \dots, x_0 = t_0)$ satisfying properties (\ref{cond1}) and (\ref{cond2}). Clearly $v \not \in V(T')$, because $s > s'$. Concatenating both paths, we have that $T =  (v,u) \cup T'$ has length $s$ and satisfies (\ref{cond1}) e (\ref{cond2}).
    
    \item If $v \in N_{3i+2}$, for some $i$, then there is $u \in N'_{3i+1}$ such that $(v,u) \in A(D)$. Note that $u \not \in N_{3i+1}$ since it would contradict the $3$-independence of the $3$-kernel $K$. Since $u \in N'_{3i+1}$, there exists $w \in N_{3i}$ such that $(u,w) \in A(D)$. Note that the path $(v,u,w)$ satisfies both properties (\ref{cond1}) and (\ref{cond2}). From the induction hypothesis, there is a path $T' = (t_{s'} = w, \dots, t_0 = x_0)$ satisfying (\ref{cond1}) and (\ref{cond2}). From the definition of each set in the $3$-substitution sequence, clearly $v,u \not \in T'$, because $3i+2 > 3i+1 > 3i$. Concatenating both paths, we have that $T =  (v,u,w) \cup T'$ has length $s$ and satisfies (\ref{cond1}) and (\ref{cond2}).
    
    \item If $v \in N_{3i}$, for some $i$, then there is $u \in N'_{3(i-1)+2} \cup N_{3(i-1)+2}$ such that $(v,u) \in A(D)$. If there exists $u \in N_{3(i-1)+2}$, then, from the induction hypothesis, there is a path $T' = (t_{s'} = u, \dots, t_0 = x_0)$ satisfying (\ref{cond1}) and (\ref{cond2}). Concatenating both paths, we have that $T =  (v,u) \cup T'$ has length $s$ and satisfies (\ref{cond1}) and (\ref{cond2}). If does not exist $u \in N_{3(i-1)+2}$ such that $(v,u) \in A(D)$, then there is $w \in N_{3(i-1)+1}$. Let $T'' = (v,u,w)$. From the induction hypothesis, there is a path $T' = (t_{s'} = w, \dots, t_0 = x_0)$ satisfying (\ref{cond1}) and (\ref{cond2}). Concatenating both paths, we have that $T =  T'' \cup T'$ has length $s$ and satisfies (\ref{cond1}) and (\ref{cond2}).
\end{itemize} We now prove (\ref{cond3}) is satisfied. 

Let $T = (t_s, \dots, t_0)$ be a path satisfying properties (\ref{cond1}) and (\ref{cond2}). Assume, for the sake of contradiction, that there exists $(t_i,t_{i-2})$ in $D$, such that $t_i \not \in N'_{3k+1}$ or $t_{i-2} \not \in N'_{3(k-1)+2}$, for some $k > 0$ and $i>2$. Let's analyze each case and conclude it is an absurd the existence of such chord.

\begin{itemize}
    \item Assume that $t_{i-2} \in N_{3k}$, for some $k$. By the definition of $N_{3k+2}$ and $N'_{3k+2}$, $t_i \not \in N_{3k+2} \cup N'_{3k+2}$, since $d(t_i,t_{i-2}) = 1$. This contradicts property (\ref{cond1}). Figure \ref{fig:casoa} depicts this case.
    
    \begin{figure}
    \centering
    \begin{tikzpicture}
                \node[vertex] (-2) {};
                \node[vertex, left of = -2] (-1) {};
                \node[vertex, left of = -1] (0) {};
                
                \node[above of = -2, yshift=-0.2cm] {$t_{i-2} \in N_{3k}$};
                \node[above of = 0, yshift=-0.2cm] {$t_{i} \in N_{3k+2} \cup N'_{3k+2}$};
                
                \node[right of = -2] (rd) {$\dots$};
                \node[left of = 0] (ld) {$\dots$};
                
                \draw[directed] (ld) to (0);
                \draw[directed] (-2) to (rd);

                \draw[directed] (0) to (-1);
                \draw[directed] (-1) to (-2);
                \draw[directed, bend right = 40] (0) to (-2);
                
    \end{tikzpicture}
    \caption{Example of the case where $t_{i-2} \in N_{3k}$.}
    \label{fig:casoa}
\end{figure}
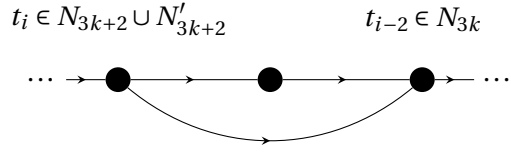

    \item Assume that $t_{i-2} \in N_{3k+1} \cup N'_{3k+1}$, for some $k$. The vertex $t_{i-3}$ exists in $T$, because we know it ends in $t_0 \in N_0$. It follows from (\ref{cond2}) that $t_{i-3} \in N_3k$ and $t_i \in N_{3(k+1)}$. The path $(t_i,t_{i-2}, t_{i-3})$ contradicts Lemma \ref{lema:kkernelabsorventequaseindependente} since it is a $(N_{3(k+1)}, N_{3k})$-path with length two. Therefore, $t_{i-2} \not \in N_{3k+1}$. Figure \ref{fig:casob} depicts this case.
    
    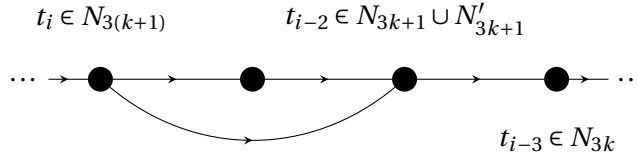
\begin{figure}
    \centering
    \begin{tikzpicture}
                \node[vertex] (-2) {};
                \node[vertex, left of = -2] (-1) {};
                \node[vertex, left of = -1] (0) {};
                \node[vertex, right of = -2] (-3) {};
                
                \node[above of = -2, yshift=-0.2cm] {$t_{i-2} \in N_{3k+1} \cup N'_{3k+1}$};
                \node[above of = 0, yshift=-0.2cm] {$t_{i} \in N_{3(k+1)}$};
                \node[below of = -3, yshift=0.2cm] {$t_{i-3} \in N_{3k}$};
                
                \node[right of = -3] (rd) {$\dots$};
                \node[left of = 0] (ld) {$\dots$};
                
                \draw[directed] (ld) to (0);
                \draw[directed] (-3) to (rd);

                \draw[directed] (0) to (-1);
                \draw[directed] (-1) to (-2);
                \draw[directed] (-2) to (-3);
                \draw[directed, bend right = 40] (0) to (-2);
                
    \end{tikzpicture}
    \caption{Example of the case where $t_{i-2} \in N_{3k+1}$ or $N'_{3k+1}$.}
    \label{fig:casob}
\end{figure}

    \item Assume that $t_{i-2} \in N_{3k+2}$, for some $k$. If $i = s$, then $t_i \in N_{3(k+1)+1}$, which contradicts the $3$-independence of the $3$-kernel $K$. So, $t_i \not \in N_{3(k+1)+1}$ and there is $t_{i+1} \in V(T)$. From (\ref{cond1}), since $t_i \not \in N_{3(k+1)+1}$, then $t_{i+1} \in N_{3(k+1)+2}$. But then $d(t_{i+1},t_{i-2}) = 2$ and it contradicts the $3$-independence of $K$. Figure \ref{fig:casoc} depicts this case.
    
    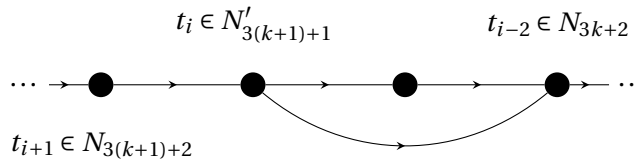
\begin{figure}
    \centering
    \begin{tikzpicture}
                \node[vertex] (-2) {};
                \node[vertex, left of = -2] (-1) {};
                \node[vertex, left of = -1] (0) {};
                \node[vertex, left of = 0] (1) {};
                
                \node[above of = -2, yshift=-0.2cm] {$t_{i-2} \in N_{3k+2}$};
                \node[above of = 0, yshift=-0.2cm] {$t_i \in N'_{3(k+1)+1}$};
                \node[below of = 1, yshift=0.2cm] {$t_{i+1} \in N_{3(k+1)+2}$};
                
                \node[right of = -2] (rd) {$\dots$};
                \node[left of = 1] (ld) {$\dots$};
                
                \draw[directed] (ld) to (1);
                \draw[directed] (-2) to (rd);

                \draw[directed] (0) to (-1);
                \draw[directed] (-1) to (-2);
                \draw[directed] (1) to (0);
                \draw[directed, bend right = 40] (0) to (-2);
                
    \end{tikzpicture}
    \caption{Example of the case where $t_{i-2} \in N_{3k+2}$.}
    \label{fig:casoc}
\end{figure}

    \item From the last cases, $t_{i-2} \in N'_{3k+2}$, for some $k$. From (\ref{cond1}), $t_i \in N_{3(k+1)+1}$ or $N'_{3(k+1)+1}$ and $t_{i-3} \in N_{3k+1}$. By our assumption, $t_i \in N_{3(k+1)+1}$. This contradicts the $3$-independence of the $3$-kernel $K$, because $d(t_i,t_{i-3}) = 2$. Figure \ref{fig:casod} depicts this case.
    
    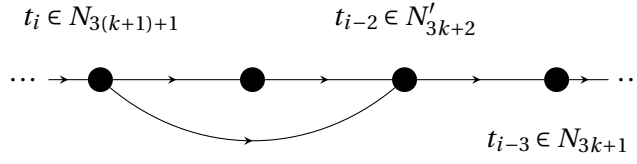
\begin{figure}
    \centering
    \begin{tikzpicture}
                \node[vertex] (-2) {};
                \node[vertex, left of = -2] (-1) {};
                \node[vertex, left of = -1] (0) {};
                \node[vertex, right of = -2] (-3) {};
                
                \node[above of = -2, yshift=-0.2cm] {$t_{i-2} \in N'_{3k+2}$};
                \node[above of = 0, yshift=-0.2cm] {$t_{i} \in N_{3(k+1)+1}$};
                \node[below of = -3, yshift=0.2cm] {$t_{i-3} \in N_{3k+1}$};
                
                \node[right of = -3] (rd) {$\dots$};
                \node[left of = 0] (ld) {$\dots$};
                
                \draw[directed] (ld) to (0);
                \draw[directed] (-3) to (rd);

                \draw[directed] (0) to (-1);
                \draw[directed] (-1) to (-2);
                \draw[directed] (-2) to (-3);
                \draw[directed, bend right = 40] (0) to (-2);
                
    \end{tikzpicture}
    \caption{Example of the case where $t_i \in N_{3(k+1)+1}$.}
    \label{fig:casod}
\end{figure}
    
\end{itemize} We conclude that if $(t_i,t_{i-2}) \in A(D)$, then $t_i \in N'_{3(k+1)+1}$ and $t_{i-2} \in N'_{3k+2}$, for some $k$. $\blacksquare$
\end{proof}

\begin{lema}
\label{lema:umaunicacordadessetipo}
Let $D$ be a digraph, let $S = (N_0,N_1,\dots,N_{3p})$ be a $3$-substitution sequence starting at $x_0 \in V(D)$ and the $3$-kernel $K\subseteq D-x_0$, and let $(N'_1,N'_2, \dots, N'_{3p-1})$ be the sequence of intermediate sets of $S$. Let $T = (t_s, \dots, t_0)$ be a road of $D$. If $(t_i,t_{i-2}) \in A(D)$, for some $2 < i <s$, then there does not exist $(t_{i'},t_{i'-2}) \in A(D)$ for $i' \neq i$. Furthermore, $t_{3k'+2} \in N_{3k+2}$ for every $3k'+2 > i$.
\end{lema}

\begin{proof}
Assume, for the sake of contradiction, that $(t_{i'},t_{i'-2}) \in A(D)$ for $i' \neq i$. Let $j$ be the smallest integer such that $(t_{j},t_{j-2}) \in A(D)$ and let $j'$ be the smallest integer greater than $j$ such that $(t_{j'},t_{j'-2}) \in A(D)$. Consider the path $(t_{j'}, t_{j'-1}, t_{j'-2}, \dots, t_{j}, t_{j-1},t_{j-2})$, illustrated by Figure \ref{fig:caminho2cordas}. From property (\ref{cond3}) of $T$, $t_{j-2} \in N'_{3k+2}$ and $t_{j} \in N'_{3(k+1)+1}$, for some $k$. We now prove that for no $\ell > j$ is true that $t_{\ell} \in N'_{3k'+2}$, for some $k'$. From property (\ref{cond1}) of $T$ we have that $t_{j+1} \in N_{3(k+1)+2}$. Assume that $t_{j+4} \in N'_{3(k+2)+2}$. Then $t_{j+3} \in N_{3(k+2)+1}$, by property (\ref{cond1}) of $T$. This contradicts the $3$-independence of the $3$-kernel $K$, since $d(t_{j+3},t_{j+1}) = 2$. Repeating the same argument, we conclude that $t_{j'-2} \not \in N'_{3k'+2}$,  for any $k'$. This contradicts property (\ref{cond3}) of $T$. Therefore, if there exists a short chord in $T$, it is unique. $\blacksquare$
\end{proof}

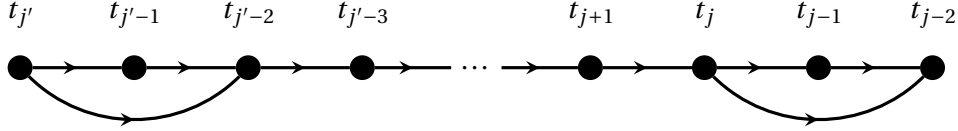
\begin{figure}
    \centering
    \begin{tikzpicture}[line width = 1.2,vertex/.append style={node distance=1.5cm}]
                \node[vertex] (0) {};
                \node[vertex, left of = 0] (1) {};
                \node[vertex, left of = 1] (2) {};
                \node[vertex, left of = 2] (3) {};
                \node[left of = 3, xshift=-0.5cm] (a) {$\dots$};
                \node[vertex, left of = a] (4) {};
                \node[vertex, left of = 4] (5) {};
                \node[vertex, left of = 5] (6) {};
                \node[vertex, left of = 6] (7) {};
                
                \draw[directed] (7) to (6);
                \draw[directed] (6) to (5);
                \draw[directed] (5) to (4);
                \draw[directed] (4) to (a);
                \draw[directed] (a) to (3);
                \draw[directed] (3) to (2);
                \draw[directed] (2) to (1);
                \draw[directed] (1) to (0);
                
                \node[above of = 0, yshift=-0.3cm] {$t_{j-2}$};
                \node[above of = 1, yshift=-0.3cm] {$t_{j-1}$};
                \node[above of = 2, yshift=-0.3cm] {$t_{j}$};
                \node[above of = 3, yshift=-0.3cm] {$t_{j+1}$};
                \node[above of = 4, yshift=-0.3cm] {$t_{j'-3}$};
                \node[above of = 5, yshift=-0.3cm] {$t_{j'-2}$};
                \node[above of = 6, yshift=-0.3cm] {$t_{j'-1}$};
                \node[above of = 7, yshift=-0.3cm] {$t_{j'}$};
                
                \draw[directed, bend right = 45] (7) to (5);
                \draw[directed, bend right = 45] (2) to (0);

    \end{tikzpicture}
    \caption{A path $(t_{j'},\dots, t_{j-2})$ and its two short chords.}
    \label{fig:caminho2cordas}
\end{figure}

\subsection{Applying the $3$-substitution method}

In this section, we prove that a strongly connected digraph whose every circuit of length not dividable by three has four short chords has a $3$-kernel. The strategy employed in the demonstration is the following. We prove a lemma that guarantees the existence of a minimal $(N_0,N_i)$-path whose length is equivalent to the additive inverse of $i$ modulo three, for every $i <3p$ with $i \neq 1$. In other words, the sum of the length of the path with the index of the set results in a number dividable by three. The concatenation of this path with the path from Lemma \ref{lema:kkernelcaminhopropriedades} results, therefore, in a circuit of length dividable by three. We then use this fact to show that there does not exist a path with length at most two between vertices in the pre-$3$-kernel. First, we prove the lemma.

\begin{lema}
\label{lema:kkernelcaminhoinversoaditivo}
Let $D$ be a digraph, let $S = (N_0,N_1,\dots,N_{3p})$ be a sequence of $3$-substitution of $D$ starting at $x_0 \in V(D)$ and the $3$-kernel $K\subseteq D-x_0$, and let $(N'_1,N'_2, \dots, N'_{3p-1})$ be the sequence of intermediate sets of $S$. Let $T = (t_{s'}, \dots, t_0)$ be a road of $D$. Assume that every circuit $C$ of $D$ such that $|C| \not \equiv 0 \Mod 3$ has four short chords. Then for every integer $s \leq s'$, $s \neq 1$, there exists a minimal $(x_0,t_{s})$-path of length $\equiv -s \Mod 3$ in $D$.
\end{lema}

\begin{proof}
Let $W = (w_0 = t_0, \dots, w_n = t_s)$ be a minimal $(t_0,t_s)$-path. Such path exists because $D$ is strongly connected. Let $Z = (z_{0}, \dots, z_{\ell})$ be the sequence, ordered by order of appearance in $W$, of the vertices in $W$ that intercept $T$, with exception of $t_1$. Figure \ref{fig:Z} depicts an example where $T$ and $W$ intercept. Note that both paths have, at leat, two vertices in common: $w_0=t_0$ e $w_n=t_s$. Therefore, $z_0 = w_0 = t_0$ and $z_\ell = w_n = t_s$. We will show that for every $z_{\ell'}$: if $z_{\ell'} = t_k$, then the $(t_0,t_k)$-path $(t_0,W,z_{\ell'})$ has length $\equiv -k \Mod 3$. Note that, if this is true, then the circuit $(w_0,W,z_{\ell'}) \cup (z_{\ell'},T,x_0)$ has length $\equiv 0 \Mod 3$.

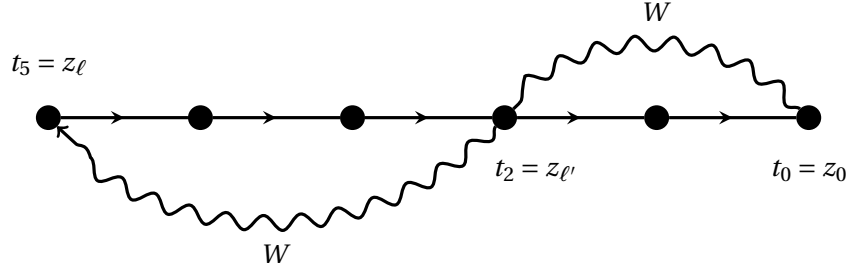
\begin{figure}
    \centering
    \begin{tikzpicture}[line width = 1.2]
                \node[vertex] (0) {};
                \node[vertex, left of = 0] (1) {};
                \node[vertex, left of = 1] (2) {};
                \node[vertex, left of = 2] (3) {};
                \node[vertex, left of = 3] (4) {};
                \node[vertex, left of = 4] (5) {};
                
                \draw[directed] (5) to (4);
                \draw[directed] (4) to (3);
                \draw[directed] (3) to (2);
                \draw[directed] (2) to (1);
                \draw[directed] (1) to (0);
                
                \draw[decorate,decoration={snake,amplitude=1mm,segment length=0.5cm,post length=1mm}, bend right = 45] (0) to (2);
                \draw[->,decorate,decoration={snake,amplitude=1mm,segment length=0.5cm,post length=1mm}, bend left = 45] (2) to (5);
                
                \node[below of = 0, yshift = 0.3cm] {$t_0 = z_0$};
                \node[below of = 2, xshift=0.4cm, yshift = 0.3cm] {$t_2 = z_{\ell'}$};
                \node[above of = 5,  yshift = -0.3cm] {$t_5 = z_\ell$};
                
                \node[above of = 1, yshift=0.4cm] {$W$};
                \node[below of = 4, yshift=-0.8cm, xshift=1cm] {$W$};

    \end{tikzpicture}
    \caption{Example of paths $T$ and $W$. We have $\ell' = 1$ and $\ell = 2$.}
    \label{fig:Z}
\end{figure}

\textbf{Base case}: ($\ell' = 0$) Clearly, $z_0 = t_0 = w_0$. Since $d(t_0,t_0) = 0$, the result follows.

\textbf{Induction hypothesis}: For every $\ell'' < \ell'$, a minimal $(t_0,t_k)$-path $(t_0,W,z_{\ell''} = t_k)$ has length $\equiv -k \Mod 3$.

\textbf{Inductive step}: Let $z_{\ell'} = w_{n'} = t_{r'}$ and let $z_{\ell''} = w_{n''} = t_{r''}$ be the intersection in $Z$ preceding $z_{\ell'}$ such that $d(z_{\ell''}, z_{\ell'})>1$. We will show that if $(w_0,W,w_{n'} = z_{\ell'})$ does not have length $ \equiv -r' \Mod 3$, then $(w_{n''} = z_{\ell''},W, w_{n'} = z_{\ell'}) \cup (t_{r'} =  z_{\ell'},T,t_{r''}= z_{\ell''})$ is a circuit of length not dividable by three that can not have four short chords. Note that if we chose $z_{\ell''}$ such that $d(z_{\ell''}, z_{\ell'})=1$, then the circuit $(w_{n''} = z_{\ell''},W, w_{n'} = z_{\ell'}) \cup (t_{r'},T,t_{r''})$ is a cycle of length two if $d(z_{\ell'}, z_{\ell''})=1$. This cycle, although it does not have four short chords, is not a contradiction. Figure \ref{fig:casomeme} depicts an example of such case. We have two cases:

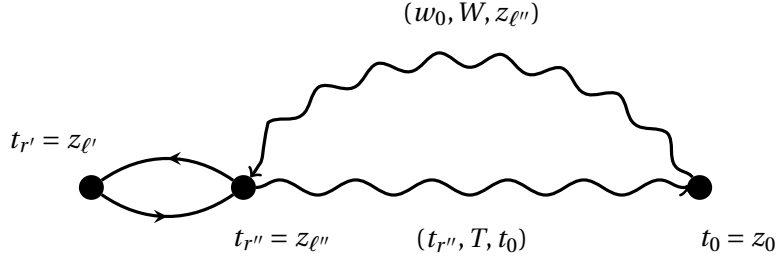
\begin{figure}
    \centering
    \begin{tikzpicture}[line width = 1.2]
                \node[vertex] (0) {};
                \node[vertex, left of = 0, xshift=-4cm] (1) {};
                \node[vertex, left of = 1] (2) {};

                \draw[->,decorate,decoration={snake,amplitude=1mm,segment length=0.8cm,post length=1mm}, bend right = 55] (0) to (1);
                \draw[directed, bend right = 35] (1) to (2);
                
                \draw[->,decorate,decoration={snake,amplitude=1mm,segment length=1cm,post length=0.3mm}] (1) to (0);
                \draw[directed, bend right = 35] (2) to (1);
                
                \node[below of = 0, xshift=0.5cm, yshift = 0.3cm] {$t_0 = z_0$};
                \node[below of = 1, xshift=0.5cm, yshift = 0.3cm] {$t_{r''} = z_{\ell''}$};
                \node[above of = 2, xshift=-0.5cm, yshift = -0.4cm] {$t_{r'} = z_{\ell'}$};
                
                \node[below of = 1, xshift=3cm, yshift = 0.3cm] {$(t_{r''},T,t_0)$};
                
                \node[below of = 1, xshift=3cm, yshift = 3.3cm] {$(w_{0},W,z_{\ell''})$};

    \end{tikzpicture}
    \caption{Example of the case where $(w_{n''} = z_{\ell''},W, w_{n'} = z_{\ell'}) \cup (t_{r'} =  z_{\ell'},T,t_{r''}= z_{\ell''})$ is a cycle of length two.}
    \label{fig:casomeme}
\end{figure}

\begin{itemize}
    \item $r' > r''$, that is, $t_{r'}$ precedes $t_{r''}$ in $T$. Consider the path $(w_0, \dots, w_n' = t_r')$. Assume, for the sake of contradiction, that $n' + r' \not \equiv 0 \Mod 3$. Consider the closed walk $(w_0, W, w_{n'}) \cup (t_{r'},T,t_0)$ of length $n' + r'$. It follows from the induction hypothesis that the path $(w_0,W,w_{n''})$ has length $n'' \equiv -r'' \Mod 3$. Consider the circuit $C = (w_{n''} = t_{r''}, W, w_{n'} = t_{r'})  \cup (t_{r'},T, t_{r''})$ of length $n' - n'' + r' - r''$. Since $n'' + r'' \equiv 0 \Mod 3$, we have that $|C| \equiv n' + r' \not \equiv 0 \Mod 3$. Figure \ref{fig:cas1} depicts an example of paths $(w_{0},W,z_{\ell''})$ and $(z_{\ell''},W,z_{\ell'})$. Therefore, $C$ must have four short chords. From the minimality of $W$, the only possible short chords in $C$ are $(t_{r''+1}, w_{n''+1})$, $(w_{n'-1}, t_{r'-1})$ and $(t_i, t_{i-2})$, for some $r''+2 < i < r'$. From Lemma \ref{lema:umaunicacordadessetipo}, the chord in form $(t_i, t_{i-2})$ is unique in $T$. Hence, the circuit can only have three short chords, a contradiction. So $(w_0, W, w_{n'} = t_{r'})$ has length $\equiv -r' \Mod 3$.
    
    \begin{figure}
    \centering
    \begin{tikzpicture}[line width = 1.2]
                \node[vertex] (0) {};
                \node[vertex, left of = 0, xshift=-4cm] (1) {};
                \node[vertex, left of = 1, xshift=-4cm] (2) {};

                \draw[->,decorate,decoration={snake,amplitude=1mm,segment length=0.8cm,post length=1mm}, bend right = 55, blue] (0) to (1);
                \draw[->,decorate,decoration={snake,amplitude=1mm,segment length=0.8cm,post length=1mm}, bend left = 55, red] (1) to (2);
                
                \draw[->,decorate,decoration={snake,amplitude=1mm,segment length=1cm,post length=0.3mm}, blue] (1) to (0);
                \draw[->,decorate,decoration={snake,amplitude=1mm,segment length=1cm,post length=0.3mm}, red] (2) to (1);
                
                \node[below of = 0, xshift=0.5cm, yshift = 0.3cm] {$t_0 = w_0 = z_0$};
                \node[below of = 1, xshift=0.5cm, yshift = 0.3cm] {$t_{r''} = z_{\ell''}$};
                \node[above of = 2, xshift=-0.5cm, yshift = -0.4cm] {$t_{r'} = w_{n'} = z_{\ell'}$};
                
                \node[below of = 1, xshift=3cm, yshift = 0.3cm] {$(t_{r''},T,t_0)$};
                \node[above of = 2, xshift=3cm, yshift = -0.4cm] {$(t_{r'},T,t_{r''})$};
                
                \node[below of = 1, xshift=3cm, yshift = 3.3cm] {$(w_{0},W,z_{\ell''})$};
                \node[above of = 2, xshift=3cm, yshift = -3.4cm] {$(z_{\ell''},W,z_{\ell'})$};

    \end{tikzpicture}
    \caption{An example depicting paths $(w_{0},W,z_{\ell''})$ and $(z_{\ell''},W,z_{\ell'})$. It follows from the induction hypothesis that the closed walk in blue has length $|(w_{0},W,z_{\ell''}) \cup (t_{r''},T,t_0)| \equiv 0 \Mod 3$. If $|W \cup T| \not \equiv 0 \Mod 3$, then circuit $C$, in red, has length $|(z_{\ell''},W,z_{\ell'}) \cup (t_{r'},T,t_{r''})| \not \equiv 0 \Mod 3$.}
    \label{fig:cas1}
\end{figure}
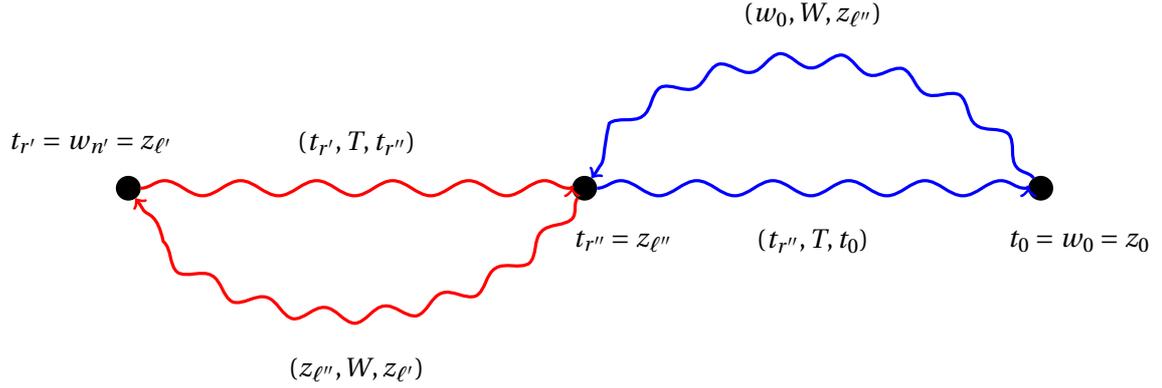

    \item $r'' > r'$, that is, $t_{r''}$ precedes $t_{r'}$ in $T$. Consider the path $(w_0, \dots, w_n' = t_r')$. Assume, for the sake of contradiction, that $n' + r' \not \equiv 0 \Mod 3$. Let $r'''$ be the greatest integer less than $r'$ such that $t_{r'''} = z_{\ell'''}$, for some $\ell'''$. Let $w_{n'''} = t_{r'''}$. Consider the closed walk $(w_0, W, w_{n'} = t_{r'}) \cup (t_{r'}, T, t_0)$ of length $n' + r'$. It follows from the induction hypothesis that $n''' \equiv -r''' \Mod 3$. Let $C = (w_{n'''}, W, w_{n'} = t_{r'}) \cup (t_{r'},T,t_{r'''})$ be a circuit with length $n' - n''' + r' - r'''$. Since $n''' + r''' \equiv 0 \Mod 3$, then $|C| \equiv n' + r' \not \equiv 0 \Mod 3$. Figure \ref{fig:casb} depicts an example of $C$. Therefore, $C$ must have four short chords. From the minimality of $W$, the only possible short chords in $C$ are $(t_{r'''+1}, w_{n'''+1})$, $(w_{n'-1}, t_{r'-1})$ and $(t_i, t_{i-2})$, for some $r'''+2 < i < r'$. From Lemma \ref{lema:umaunicacordadessetipo}, the chord in form $(t_i, t_{i-2})$ is unique in $T$. Hence, the circuit can only have three short chords, a contradiction. So, $(w_0, \dots, w_{n'} = t_{r'})$ has length $\equiv -r' \Mod 3$. $\blacksquare$
    
    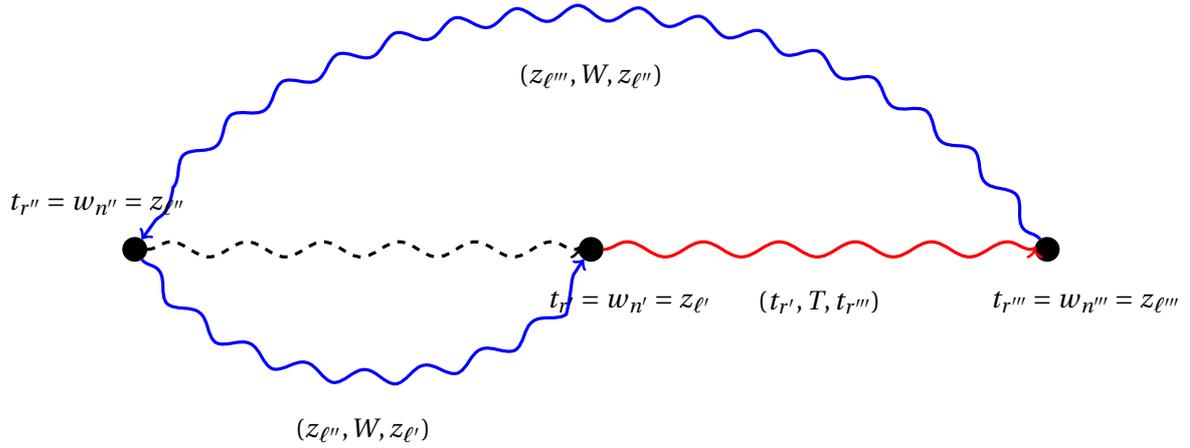
\begin{figure}
    \centering
    \begin{tikzpicture}[line width = 1.2]
                \node[vertex] (0) {};
                \node[vertex, left of = 0, xshift=-4cm] (1) {};
                \node[vertex, left of = 1, xshift=-4cm] (2) {};

                \draw[->,decorate,decoration={snake,amplitude=1mm,segment length=0.8cm,post length=1mm}, bend right = 55, blue] (0) to (2);
                \draw[->,decorate,decoration={snake,amplitude=1mm,segment length=0.8cm,post length=1mm}, bend right = 55, blue] (2) to (1);
                
                \draw[->,decorate,decoration={snake,amplitude=1mm,segment length=1cm,post length=0.3mm}, red] (1) to (0);
                \draw[->,decorate,decoration={snake,amplitude=1mm,segment length=1cm,post length=0.3mm}, black, dashed] (2) to (1);
                
                \node[below of = 0, xshift=0.5cm, yshift = 0.3cm] {$t_{r'''}  = w_{n'''} = z_{\ell'''}$};
                \node[below of = 1, xshift=0.5cm, yshift = 0.3cm] {$t_{r'} = w_{n'} = z_{\ell'}$};
                \node[above of = 2, xshift=-0.5cm, yshift = -0.4cm] {$t_{r''} =  w_{n''} = z_{\ell''}$};
                
                \node[below of = 1, xshift=3cm, yshift = 0.3cm] {$(t_{r'},T,t_{r'''})$};
                
                \node[below of = 1, yshift = 3.3cm] {$(z_{\ell'''},W,z_{\ell''})$};
                \node[above of = 2, xshift=3cm, yshift = -3.4cm] {$(z_{\ell''},W,z_{\ell'})$};

    \end{tikzpicture}
    \caption{An example depicting $C$. When concatenated, paths $(z_{\ell'''},W,z_{\ell''})$ and $(t_{r'},T,t_{r'''})$ form a circuit of length $n' - n''' + r' - r'''$. Since $n''' + r''' \equiv 0 \Mod 3$, then $|C| \equiv n' + r' \not \equiv 0 \Mod 3$.}
    \label{fig:casb}
\end{figure}

\end{itemize}

\end{proof}

Next, we present the definition of a quasi-$3$-kernel-perfect digraph and a $3$-kernel-perfect digraph. After that, we present Theorem \ref{thm:meuteorema4cc}.

\begin{defn}
A digraph $D$ is called \textbf{quasi-$3$-kernel-perfect} if every induced proper subdigraph of $D$ has a $3$-kernel.
\end{defn}

\begin{defn}
A digraph $D$ is called \textbf{$3$-kernel-perfect} if every induced subdigraph of $D$ has a $3$-kernel.
\end{defn}

\begin{thm}
\label{thm:meuteorema4cc}
Let $D$ be a strongly connected digraph that is quasi-$3$-kernel-perfect. Let $x_0 \in V(D)$. If every circuit $C$ of $D$ such that $|C| \not \equiv 0 \Mod 3$ has four short chords, then $D$ is $3$-kernel-perfect.
\end{thm}

\begin{proof}
Let $K$ be a $3$-kernel of $D-x_0$, let $S = (N_0, \dots, N_{3p})$ be a $3$-substitution sequence of $D$ starting at $x_0$ and $K$, let $S'$ be its sequence of intermediate sets and let $K'$ be the pre-$3$-kernel of $D$. From Lemma \ref{lema:kkernelabsorventequaseindependente} we know that $K'$ is $2$-absorbent and, if it is not $3$-independent, there is a pair $a,b \in V(D)$ such that $a \in N_{3i}$, $b \in N_{3j}$, $i<j$ and there exists a $(a,b)$-path with length $\leq 2$. To prove that $K'$ is a $3$-kernel of $D$ suffices to show that such $(a,b)$-path does not exist in $D$.

Assume, for the sake of contradiction, that there exists a $(a,b)$-path $P$ with length at most two in $D$ and choose $a \in N_{3i}$ such that $i$ is the smallest integer possible, and then choose $b \in \Delta^2(a)$ such that $b \in N_{3j}$ with the smallest $j$ possible. From Lemma \ref{lema:kkernelcaminhopropriedades}, there exists a $(b,x_0)$-path $T=(t_s=b,\dots ,t_0=x_0)$ with length $3j$. From Lemma \ref{lema:kkernelcaminhoinversoaditivo}, there exists a minimal $(x_0, a)$-path $Z = (z_0 = x_0, \dots, z_r = a)$ with length $3n$ in $D$, for some integer $n$. Note that $T \cup Z \cup P$ is a closed walk of length $\not \equiv 0 \Mod 3$. Let $r'$ be the greatest integer such that $z_{r'} = t_{s'}$, for some integer $s'\neq 1$.

Let $T' = (t_s=b, \dots, t_{s'})$ and $Z' = (z_{r'}, \dots, z_r = a)$ be paths in $D$. From Lemma \ref{lema:kkernelcaminhoinversoaditivo}, $|(t_{s'},T,t_{0} = z_0) \cup (z_0 = t_0, Z, z_{r'})| \equiv 0 \Mod 3$. Therefore, $C = T' \cup Z' \cup P$ is a circuit of length $\equiv |T \cup Z \cup P| \not \equiv 0 \Mod 3$ and has four short chords.

Assume that $|P| = 1$. Then the chords in $C$ are: $(t_{s'+1},z_{r'+1})$, $(z_{r-1},t_s)$, $(z_r,t_{s-1})$ and $(t_i',t_{i'-2}) \in A(D)$, for some $s'+2 \leq i' \leq s$. Figure \ref{fig:exempl0} depicts $C$ and its possible short chords. Since $(t_i',t_{i'-2}) \in A(D)$, it follows from Lemma \ref{lema:umaunicacordadessetipo} that $t_{s-1} \in N_{3(j-1)+2}$. So, $z_r \not \in N_{3i}$, since $d(z_r,t_{s-1}) =1$, which contradicts the definition of $N_{3i}$. 

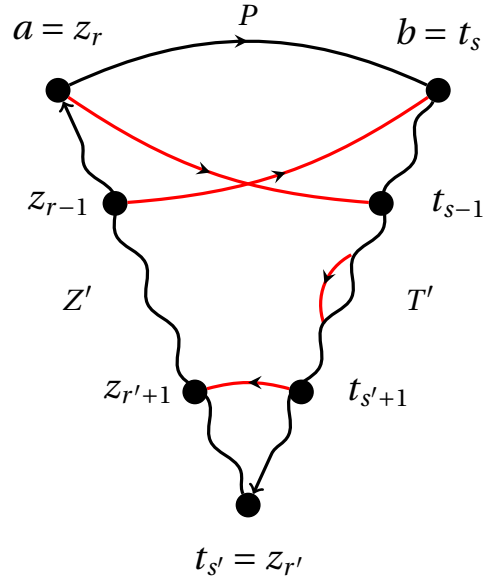
\begin{figure}
    \centering
    \begin{tikzpicture}[line width = 1.2,vertex/.append style={node distance=3cm}]
                \node[vertex] (a) {};
                \node[vertex] (b) [right of= a, xshift = 2cm] {};
                \node[vertex] (v) [below of= a, yshift = -2.5cm, xshift = 2.5cm] {};
                
                \node[vertex] (t1) [left of= v, xshift = 2.3cm, yshift=1.5cm] {};
                \node[vertex] (w1) [right of= v, xshift = -2.3cm, yshift=1.5cm] {};
                
                \node[vertex] (tr) [left of= a, xshift = 3.75cm, yshift=-1.5cm] {};
                \node[vertex] (wr) [right of= b, xshift = -3.75cm, yshift=-1.5cm] {};

                \draw[directed, bend right = 15,red]  (w1) to (t1);
                \draw[directed, bend right = 15,red]  (tr) to (b);
                \draw[directed, bend right = 15,red]  (a) to (wr);
                
                \node [left of = t1,xshift = .25cm] {\Large $z_{r'+1}$};
                \node [left of = tr,xshift = .25cm] {\Large $z_{r-1}$};
                
                \node [right of = w1] {\Large $t_{s'+1}$};
                \node [right of = wr] {\Large $t_{s-1}$};
                
                \node[below of = wr, xshift=-0.26cm, yshift=0.39cm] (meme) {};
                \node[below of =meme, xshift=-0.45cm, yshift=-0.12cm] (meme2) {};
                
                \draw[directed, bend right = 40, red] (meme) to (meme2);

                \draw[directed, bend left = 25]  (a) to (b);
                \draw[->,decorate,decoration={snake,amplitude=-1mm,segment length=1cm,post length=1mm}] (v) to (a);
                \draw[->,decorate,decoration={snake,amplitude=-1mm,segment length=1cm,post length=1mm}] (b) to (v);

                \node [above of = a,yshift = -.25cm] {\Large $a = z_r$};
                \node [above of = b, yshift = -.25cm] {\Large $b = t_{s}$};
                \node [below of = v, yshift = .25cm] {\Large $t_{s'} = z_{r'}$};
                
                \node [above of = a, yshift = 0cm, xshift= 2.5cm] {\large $P$};
                \node [below of =a, yshift = -1.75cm, xshift=.25cm] {\large $Z'$};
                \node [below of =b, yshift = -1.75cm, xshift=-.25cm] {\large $T'$};

    \end{tikzpicture}
    \caption{An example of the cycle and its possible short chords, in red.}
    \label{fig:exempl0}
\end{figure}

Assume that $|P| = 2$ and let $p_1 \in P$ be its internal vertex. Then the short chords of $C$ are: $(t_{s'+1},z_{r'+1})$, $(z_{r-1},p_1)$, $(z_r,t_{s})$, $(p_1, t_{s-1})$ and $(t_i',t_{i'-2}) \in A(D)$, for some $s'+2 \leq i' \leq s$. Since we showed that $|P| \neq 1$, $(z_r,t_{s})$ does not exist. Figure \ref{fig:exempl1} depicts $C$ and its possible short chords, while Figure \ref{fig:exempl2} depicts a particular case of $P$ which implies that the hypothesis of the theorem can not be changed to \textit{``for every cycle...''}. Analogously to the last case, since $(t_i',t_{i'-2}) \in A(D)$, it follows from Lemma \ref{lema:umaunicacordadessetipo} that $t_{s-1} \in N_{3(j-1)+2}$ and, therefore, $d(z_r,t_{s-1}) =2$. A contradiction with the definition of $N_{3i}$. We conclude that there does not exist a path $P$ in $D$ such that $|P| \leq 2$. Hence, $K'$ is a $3$-kernel of $D$. $\blacksquare$

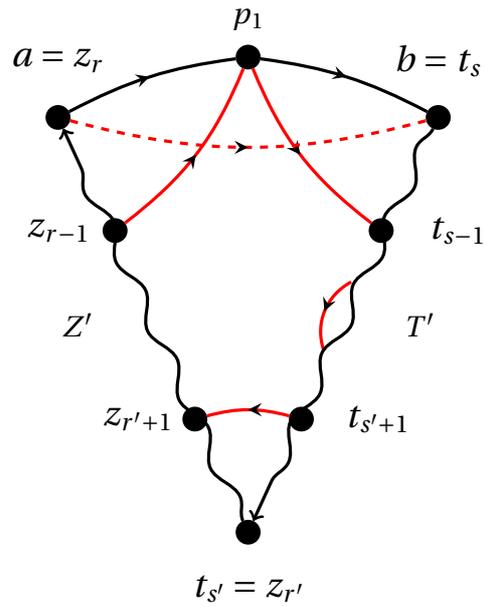
\begin{figure}
    \centering
    \begin{tikzpicture}[line width = 1.2,vertex/.append style={node distance=3cm}]
                \node[vertex] (a) {};
                \node[vertex] (b) [right of= a, xshift = 2cm] {};
                \node[vertex] (p1) [above of = a, yshift = -2.2cm, xshift= 2.5cm] {};
                \node[vertex] (v) [below of= a, yshift = -2.5cm, xshift = 2.5cm] {};
                
                \node[vertex] (t1) [left of= v, xshift = 2.3cm, yshift=1.5cm] {};
                \node[vertex] (w1) [right of= v, xshift = -2.3cm, yshift=1.5cm] {};
                
                \node[vertex] (tr) [left of= a, xshift = 3.75cm, yshift=-1.5cm] {};
                \node[vertex] (wr) [right of= b, xshift = -3.75cm, yshift=-1.5cm] {};

                \draw[directed, bend right = 15,red]  (w1) to (t1);
                \draw[directed, bend right = 15,red]  (tr) to (p1);
                \draw[directed, bend right = 15,red]  (p1) to (wr);
                \draw[directed, bend right = 15,red, dashed]  (a) to (b);
                
                \node [left of = t1,xshift = .25cm] {\Large $z_{r'+1}$};
                \node [left of = tr,xshift = .25cm] {\Large $z_{r-1}$};
                
                \node [right of = w1] {\Large $t_{s'+1}$};
                \node [right of = wr] {\Large $t_{s-1}$};
                
                \node[below of = wr, xshift=-0.26cm, yshift=0.39cm] (meme) {};
                \node[below of =meme, xshift=-0.45cm, yshift=-0.12cm] (meme2) {};
                
                \draw[directed, bend right = 40, red] (meme) to (meme2);

                \draw[directed, bend left = 10]  (a) to (p1);
                \draw[directed, bend left = 10]  (p1) to (b);
                \draw[->,decorate,decoration={snake,amplitude=-1mm,segment length=1cm,post length=1mm}] (v) to (a);
                \draw[->,decorate,decoration={snake,amplitude=-1mm,segment length=1cm,post length=1mm}] (b) to (v);

                \node [above of = a,yshift = -.25cm] {\Large $a = z_r$};
                \node [above of = b, yshift = -.25cm] {\Large $b = t_{s}$};
                \node [below of = v, yshift = .25cm] {\Large $t_{s'} = z_{r'}$};
                
                \node [above of = a, yshift = 0.3cm, xshift= 2.5cm] {\large $p_1$};
                \node [below of =a, yshift = -1.75cm, xshift=.25cm] {\large $Z'$};
                \node [below of =b, yshift = -1.75cm, xshift=-.25cm] {\large $T'$};

    \end{tikzpicture}
    \caption{An example of the cycle and its possible short chords, in red. The dashed chord $(a,b)$, can not exist.}
    \label{fig:exempl1}
\end{figure}

\begin{figure}
    \centering
    \begin{tikzpicture}[line width = 1.2,vertex/.append style={node distance=3cm}]
                \node[vertex] (a) {};
                \node[vertex] (b) [right of= a, xshift = 2cm] {};
                %\node[vertex] (p1) [above of = a, yshift = -2.2cm, xshift= 2.5cm] {};
                \node[vertex] (v) [below of= a, yshift = -2.5cm, xshift = 2.5cm] {};
                
                %\node[vertex] (t1) [left of= v, xshift = 2.3cm, yshift=1.5cm] {};
                %\node[vertex] (w1) [right of= v, xshift = -2.3cm, yshift=1.5cm] {};

                \node[vertex] (tr) [left of= a, xshift = 3.75cm, yshift=-1.5cm] {};
                %\node[vertex] (wr) [right of= b, xshift = -3.75cm, yshift=-1.5cm] {};

                %\draw[directed, bend right = 15,red]  (w1) to (t1);
                %\draw[directed, bend right = 15,red]  (tr) to (p1);
                %\draw[directed, bend right = 15,red]  (p1) to (wr);
                %\draw[directed, bend right = 15,red, dashed]  (a) to (b);
                
                %\node [left of = t1,xshift = .25cm] {\Large $z_{r'+1}$};
                \node [left of = tr,xshift = -.5cm] {\Large $z_{r-1} = p_1$};
                
                %\node [right of = w1] {\Large $t_{s'+1}$};
                %\node [right of = wr] {\Large $t_{s-1}$};
                
                %\node[below of = wr, xshift=-0.26cm, yshift=0.39cm] (meme) {};
                %\node[below of =meme, xshift=-0.45cm, yshift=-0.12cm] (meme2) {};
                
                %\draw[directed, bend right = 40, red] (meme) to (meme2);

                \draw[directed, bend left = 10]  (a) to (tr);
                \draw[directed, bend left = 10]  (tr) to (a);
                \draw[directed, bend left = 10]  (tr) to (b);
                \draw[->,decorate,decoration={snake,amplitude=-1mm,segment length=1cm,post length=1mm}] (v) to (tr);
                \draw[->,decorate,decoration={snake,amplitude=-1mm,segment length=1cm,post length=1mm}] (b) to (v);

                \node [above of = a,yshift = -.25cm] {\Large $a = z_r$};
                \node [above of = b, yshift = -.25cm] {\Large $b = t_{s}$};
                \node [below of = v, yshift = .25cm] {\Large $t_{s'} = z_{r'}$};
                
                %\node [above of = a, yshift = 0.3cm, xshift= 2.5cm] {\large $p_1$};
                \node [below of =a, yshift = -1.75cm, xshift=.25cm] {\large $Z'$};
                \node [below of =b, yshift = -1.75cm, xshift=-.25cm] {\large $T'$};
                
                \node [above of =v, yshift = 2cm] {\large $C$};

    \end{tikzpicture}
    \caption{An example of a specific case where $z_{r-1}=p_1$. The cycle $C=(t_{s'}, \dots, z_{r-1},t_s, \dots, t_{s'})$ has length multiple of three. But $(z_r,p_1,z_r)$, besides its length of two, can not have chords. If the hypothesis on the short chords were true only for cycles, then there would not be a contradiction in this case.}
    \label{fig:exempl2}
\end{figure}
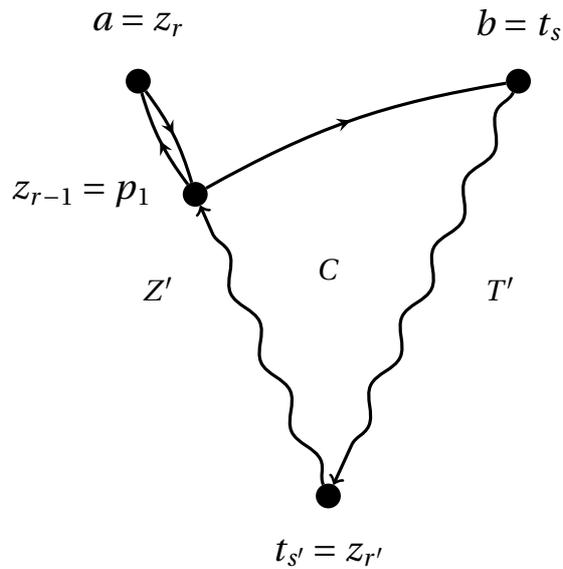
\end{proof}

%% file: 4-Section.tex
\section{Concluding remarks}

We proved two new results regarding the existence of $3$-kernels and altered the substitution method for $3$-kernels. For future works, we propose the generalization of Theorem \ref{thm:meuteorema2cc} for any $k \in \mathbb{N}$ and to study the possibility of a $4$-substitution method.

\newpage